\documentclass[12pt, reqno]{amsart}

\usepackage{xpatch}

\usepackage{amssymb}
\usepackage{amsfonts}
\usepackage{amsmath}
\usepackage{graphicx}
\usepackage{amscd}
\usepackage{enumitem}
\newtheorem{theorem}{Theorem}
\numberwithin{theorem}{section}

\numberwithin{subcase}{case}

\newtheorem{corollary}[theorem]{Corollary}

\newtheorem{definition}[theorem]{Definition}

\newtheorem{lemma}[theorem]{Lemma}

\newtheorem{question}[theorem]{Question}

\DeclareMathOperator{\cl}{cl}

\begin{document}

\title{Diagonals separating the square of a continuum}

\author[\tiny A. Illanes]{Alejandro Illanes}
\author[V. Mart\'inez-de-la-Vega]{Ver\'onica Mart\'inez-de-la-Vega}
\author[J. M. Mart\'inez-Montejano]{Jorge M. Mart\'{\i}nez-Montejano}
\author[D. Michalik]{Daria Michalik}

\address[A. Illanes]{Instituto de Matem\'{a}ticas, Universidad
Nacional Aut\'{o}noma de M\'{e}xico, Circuito Exterior, Cd. Universitaria,
Ciudad de M\'{e}xico, 04510, M\'{e}xico. ORCID 0000-0002-7109-4038.} 
\email{illanes@matem.unam.mx}

\address[V. Mart\'{\i}nez-de-la-Vega]{Instituto de Matem\'{a}ticas, Universidad Nacional Aut\'{o}noma de M\'{e}xico, Circuito Exterior, Cd. Universitaria, Ciudad de M\'{e}xico, 04510, M\'{e}xico. ORCID 0000-0002-1694-6847.} 
\email{vmvm@matem.unam.mx}

\address[J. M. Mart\'{\i}nez-Montejano]{Departamento de Matem\'{a}ticas, Facultad de Ciencias, Universidad Nacional Aut\'{o}noma de M\'{e}xico, Circuito Exterior, Cd. Universitaria,  Ciudad de M\'{e}xico, 04510, M\'{e}xico. ORCID 0000-0001-5859-8728.} 
\email{jorgemm@ciencias.unam.mx}

\address[D. Michalik]{Jan Kochanowski University, \'Swi\c etokrzyska 15, 25-406, Kielce, Poland. ORCID 0000-0001-9524-5403.} 
\email{daria.michalik@ujk.edu.pl}



\let\thefootnote\relax\footnote{This paper was partially supported by the project ``Teor\'{\i}a de Continuos e Hiperespacios, dos" (AI-S-15492) of CONACyT. Also, the work of the last author was supported by the National Science Centre, Poland, through the Grant Nr rej.: 2021/05/X/ST1/00357.}

\begin{abstract}
A metric continuum $X$ is indecomposable if it cannot be put as the union of two of its proper subcontinua. A subset $R$ of $X$ is said to be continuumwise connected provided that for each pair of points $p,q\in R$, there exists a subcontinuum $M$ of $X$ such that $\{p,q\}\subset M\subset R$. Let $X^{2}$ denote the Cartesian square of $X$ and $\Delta$ the diagonal of $X^{2}$. In \cite{ka} it was asked if for a continuum $X$, distinct from the arc, $X^{2}\setminus \Delta$ is continuumwise connected if and only if $X$ is decomposable. In this paper we show that no implication in this question holds. For the proof of the non-necessity, we use the dynamic properties of a suitable homeomorphism of the Cantor set onto itself to construct an appropriate indecomposable continuum $X$.
\end{abstract}

\maketitle

\section{Introduction}

A {\it continuum} is a compact connected non-degenerate metric space. A subcontinuum
of a continuum $X$ is a nonempty connected closed subset of $X$, so singletons
are subcontinua.
An {\it arc} is a continuum homeomorphic to the interval $[0,1]$. Let $X^{2}$ denote the Cartesian square of $X$ and $\Delta$ the diagonal
of $X^{2}$.

The authors studied in \cite{immm} conditions under which $\Delta$ satisfies some of the properties described in \cite[Table 1]{bpv} and \cite[Definition 1.1]{mm}, for a number of examples and families of continua.

Recently, H. Katsuura \cite{ka} proved that the arc is the only continuum for which $\Delta$ does not satisfy that $X^{2}\setminus \Delta$  is connected, and he included the following question \cite[p. 4]{ka} (Katsuura mentioned that this question was suggested by Wayne Lewis of Texas Tech University in a private conversation):
\begin{question}\label{dos}
If $X$ is a continuum other than the arc, is $X^{2}\setminus \Delta$ continuumwise connected if and only if $X$ is decomposable?
\end{question}
%Question \ref{dos}, can be rephrased as: If $X$ is a continuum other than the arc, is $\Delta$ not a weak cut  continuum of $X^{2}$ if and only if $X$ is decomposable?

As an application of Dynamic Systems to Continuum Theory, we use the dynamic properties of a particular homeomorphism from the Cantor set onto itself to prove that the necessity in Question 2 is not satisfied. In fact, we show that no implication in Question \ref{dos} is satisfied.

\section{No Sufficiency}
A {\it map} is a continuous function. Given continua $X$ and $Y$, and a number $\varepsilon>0$, an $\varepsilon$-{\it map} is an onto map $f:X\rightarrow Y$ such that  for each $y\in Y$, diameter($f^{-1}(y))<\varepsilon$.
The continuum $X$ is  {\it arc-like} provided that for each $\varepsilon>0$, there exists an $\varepsilon$-map $f:X\rightarrow [0,1]$. 

\begin{lemma}
\label{tres}
 Let $X$ be an arc-like continuum and $p,q\in X$ such that $p\neq q$. Then,
for every continuum $K\subset X^2$ containing the points $(p,q)$ and $(q,p)$,
$K\cap\Delta\neq\emptyset$. 
\end{lemma}
\begin{proof}
Let $d$ be a metric for $X$. Let $K$ be a continuum in $X^2$ such that $(p,q),(q,p)\in K$, for some points
$p\neq q\in X$. 
 
Suppose that $K\cap\Delta = \emptyset$. Hence, there exists $\varepsilon>0$ such
that $d(x,y)>\varepsilon$, for every point $(x,y)\in K$. Since $X$ is an
arc-like continuum, there exists an $\varepsilon$-map $\lambda:X  \rightarrow
[0,1]$. Then for each $(x,y)\in K$, $\lambda(x)\neq \lambda(y)$. By the connectedness
of $K$, we obtain that either $\lambda(x)<\lambda(y)$ for each $(x,y)\in
K$ or $\lambda(y)<\lambda(x)$ for each $(x,y)\in K$. This contradicts the
fact that $(p,q),(q,p)\in K$. 
\end{proof}

\begin{corollary}\label{cuatro}
 \label{cuatro} Let $X$ be an arc-like continuum. Then $X^{2}\setminus\Delta$ is not continuumwise connected.
\end{corollary}

Since there are decomposable arc-like continua (for example, the sin$(\frac{1}{x})$-curve), Corollary \ref{cuatro}, shows that the sufficiency in Question \ref{dos} is not satisfied. We also prove in \cite[Section 7]{immm} that, in fact the sin$(\frac{1}{x})$-curve belongs to a family of curves for which $\Delta$ satisfies a stronger property in $X^{2}$. S. B. Nadler, Jr. \cite[p. 329]{na2} called {\it Elsa continua} to the compactifications
of the ray $[0,\infty)$ whose remainder is an arc and he proved that they
are arc-like \cite[Lemma 6]{na1} (in fact, it is easy to show that a compactification
of $[0,\infty)$ with non-degenerate remainder,  is arc-like if and only
if the remainder is an arc-like continuum). Then for each Elsa continuum, $X$ is deconposable and $X^{2}\setminus \Delta$ is not continuumwise connected. 

%\begin{proposition}
%\label{cinco}
%Let $X$ be an Elsa continuum. Then $\Delta$ is a non-cut continuum in $X^{2}$, %but it is a strong center in $X^{2}$.
%\end{proposition}

%\begin{proof}
%Suppose that $X=R\cup A$, where $R$ is a ray ($R$ is homeomorphic to $[0,\infty)$) and $A$ is an arc and it is the remainder of $X$. By \cite[Theorem 3]{ka}, $\Delta$ is non-cut in $X^2$. 

%Let $I$ and $J$ be disjoint arcs in the ray $R\subset X$. Let
%$E=I\times J$ and $F=J\times I$. Then $E$ and $F$ have nonempty interior %in $X^{2}$. Let $M$ be a subcontinuum of $X^2$ such that
%$M\cap E\neq\emptyset\neq M\cap F$. Let $K=M\cup E\cup F$ and fix points %$p\in I$ and $q\in J$. Then $K$ is a subcontinuum of $X^{2}$ such that $(p,q),(q,p)\in K$. By Lemma \ref{tres}, there exists a point $z=(x,x)\in K\cap \Delta$.
%Since $(x,x)\notin E\cup F$, we have that $z\in M\cap \Delta$. We have shown that there is no a continuum in $X^{2}\setminus \Delta$ intersecting both open sets int$_{X^{2}}(E)$ and int$_{X^{2}}(F)$. Therefore $\Delta$ is a strong center in $X^2$.
%\end{proof}

\section{No Necessity}
In this section we present an indecomposable continuum $X$ such that $X^{2}\setminus \Delta$ is continuumwise connected.
%is colocally connected in $X^{2}$.
\begin{definition}\label{uno}

Let $X$ be a continuum and $A$ a subcontinuum of X
with int$_{X} (A) =\emptyset$. We say that $A$ is
a continuum of {\rm colocal connectedness} in $X$ if for each
open subset $U$ of $X$ with $A\subset U$ there exists an open subset $V$ of $X$ such that $A\subset
V \subset U$ and $X\setminus V$ is connected.
\end{definition}

By  \cite[Table 1]{bpv} and \cite[Section 1]{mm} it is known that if $\Delta$ satisfies  Definition \ref{uno} in $X^{2}$ then $X^{2}\setminus \Delta $ is continuumwise connected.

The construction and proof of the properties of $X$ strongly depends on the dynamic properties of a particular homeomorphism of the Cantor set onto itself.

Let $C$ denote the Cantor ternary set. In this section we will use a homeomorphism $f:C\rightarrow C$ such that $f$ is minimal ($C$ does not contain any proper nonempty closed subset $A$ such that $f(A)=A$, or equivalently, every orbit of $f$ is dense) and $f$ is weakly mixing (for any two nonempty subsets $U,V$ of $X^{2}$, there exist $n\in \mathbb {N}$ such that $(f\times f)^{n}(U)\cap V\neq\emptyset$). A recent reference for the existence of such homeomorphisms is \cite[Theorem 4.1]{bko}. It is known that the inverse of $f^{-1}$ is also minimal \cite[Theorem 6.2 (e)]{ac}. By \cite[Proposition 2]{dk}, $f\times f$ has a dense orbit.

We consider the space $X$ obtained by identifying in $C\times [0,1]$, for each $p\in C$, the points $(p,1)$ and $(f(p),0)$. Let $\varphi:C\times [0,1]\rightarrow X$ be the quotient mapping and let $\sigma=\varphi\times\varphi:(C\times[0,1])^{2}\rightarrow X^{2}$. Let $\rho$ be the metric on $C\times [0,1]$ given by $\rho((p,s),(q,t))=\vert p-q\vert+\vert s-t\vert)$. Fix a metric $D$ for the space $X$.

In the hypothesis of the following theorem, we write the specific properties that we use of the homeomorphism $f$.

\begin{theorem}
Suppose that $f:C\rightarrow C$ is a homeomorphism such that the orbits of $f$ and $f^{-1}$ are dense and $f\times f:C^{2}\rightarrow C^{2}$ has a dense orbit. Then $X$ is an indecomposable continuum such that the diagonal $\Delta$ in $X^{2}$ is colocally connected.  
\end{theorem}

\begin{proof} The indecomposability of $X$ is proved in \cite[Corollary, p. 552]{gu}. So we only need to show that  
$\Delta$ is colocally connected in $X^{2}$. In order to do this, take an open subset $U$ in $X^{2}$ such that $\Delta\subset U$. Then there exists $\varepsilon>0$ such that if $D(x,y)<\varepsilon$, then $(x,y)\in U$. Fix $\delta>0$ such that $\delta<\frac{1}{10}$ and if $(p,s),(q,t)\in C\times[0,1]$ and $\rho((p,s),(q,t))<2\delta$, then $D(\varphi(p,s),\varphi(q,t))<\frac{\varepsilon}{3}$ (hence $\sigma((p,s),(q,t))\in U$).

Define
$$V_{1}=\{((p,s),(q,t))\in (C\times[0,1])^{2}:\vert p-q\vert<\delta\text{ and }\vert t-s \vert<\delta\},$$ 
$$V_{2}=\{((p,s),(q,t))\in (C\times[0,1])^{2}:\vert p-f(q)\vert<\delta, s<\delta\text{ and }1-\delta<t\},$$
$$V_{3}=\{((p,s),(q,t))\in (C\times[0,1])^{2}:\vert f(p)-q\vert<\delta, t<\delta\text{ and }1-\delta<s\},$$
$$V_{0}=V_{1}\cup V_{2}\cup V_{3}$$
and
$$K=(C\times[0,1])^{2}\setminus V_{0}.$$

We check that $K$ has the following properties.

(a) $\sigma(K)$ is a continuum, and

(b) $\Delta\subset X^{2}\setminus \sigma(K)\subset U$.
 
In order to prove (a), observe that $V_{0}$ is open in $(C\times[0,1])^{2}$, so $K$ and $\sigma(K)$ are compact, so we only need to show that $\sigma(K)$ is connected. Let 
\begin{center}
$M=(C\times\{0\})\times(C\times\{\frac{1}{2}\})$. \end{center}

Notice that $M\subset K$. 

{\bf Claim 1.} Let $z=((p,s),(q,t))\in K$ be such that $s\leq t$. Then there exists a subcontinuum $A$ of $K$ such that $z,((p,0),(q,\frac{1}{2}))\in A$. 

We prove Claim 1. We consider two cases.

{\bf Case 1.} $2\delta\leq s$ or $t\leq 1-\delta$. 

For each $r\in [0,1]$, let $u(r)=rs$ and $v(r)=rt+(1-r)(t-s)$. Define $\lambda(r)=((p,u(r)),(q,v(r)))$ and $A_{1}=\{\lambda(r):r\in[0,1]\}$. Then $z,((p,0),(q,t-s))\in A_{1}$. Take $r\in [0,1]$, observe that $u(r)\leq v(r)\leq t$ and $v(r)-u(r)=t-s$. Since $u(r)\leq v(r)$ and $\delta<1-\delta$, we have that $\lambda(r)\notin V_{3}$. Since $v(r)-u(r)=t-s$ and $z\notin V_{1}$, we have that $\lambda(r)\notin V_{1}$.

In the case that $t\leq 1-\delta$, we have that $v(r)\leq 1-\delta$, Hence $\lambda(r)\notin V_{2}$. 

Now we consider the case that $2\delta\leq s$. If $r\leq \frac{1}{2}$, then $\delta\leq(1-r)2\delta\leq (1-r)s$, so $t\leq1\leq 1-\delta +(1-r)s$. This implies that $v(r)=rt+(1-r)(t-s)\leq 1-\delta$. Thus $\lambda(r)\notin V_{2}$. If $\frac{1}{2}\leq r$, then $\delta\leq\frac{s}{2}\leq rs=u(r)$. Hence $\lambda(r)\notin V_{2}$.

This completes the proof that for each $r\in [0,1]$, $\lambda(r)\in K$. Therefore $A_{1}\subset K$.

 Given $r\in [0,1]$, let $w(r)=r(t-s)+(1-r)\frac{1}{2}$ and $\eta(r)=((p,0),(q,w(r)))$. Let
 $A_{2}=
\{\eta(r):r\in[0,1]\}$. Observe that $((p,0),(q,t-s)), ((p,0),(q,\frac{1}{2}))\in A_{2}$ and $\eta(r)\notin V_{3}$.

Take $r\in [0,1]$. Since $\lambda(0)\in K$, we have that $((p,0),(q,t-s))\notin V_{1}\cup V_{2}$. Since $((p,0),(q,t-s))\notin V_{1}$, we have that either $\delta \leq \vert p-q\vert$ or $\delta\leq t-s$. If $\delta \leq \vert p-q\vert$, it is clear that $\eta(r)\notin V_{1}$. If $\delta\leq t-s$, since $\delta \leq\frac{1}{2}$, we have that $\delta \leq w(r)$. Thus $\eta(r)\notin V_{1}$. Since $((p,0),(q,t-s))\notin V_{2}$, we have that either
$\delta \leq \vert  p-f(q)\vert$ or $t\leq 1-\delta$. If $\delta \leq \vert p-f(q)\vert$,
it is clear that $\eta(r)\notin V_{2}$. If $t\leq 1-\delta$, since $\frac{1}{2}\leq 1-\delta$, we have that $w(r)\leq 1-\delta$. Thus $\eta(r)\notin V_{2}$. We have shown that for each $r\in [0,1]$, $\eta(r)\notin V_{1}\cup V_{2}\cup V_{3}$. Hence $A_{2}\subset K$. Define $A=A_{1}\cup A_{2}$. Then $A$ has the required properties.

{\bf Case 2.} $s<2\delta$ and $1-\delta<t$.

Given $r\in [0,1]$, let $y(r)=r(1-\delta)+(1-r)t$ and $\gamma(r)=((p,s),(q,y(r)))$. Let $A_{3}=\{\gamma(r):r\in [0,1]\}$. Notice that $z,((p,s),(q,1-\delta))\in A_{3}$ and for each $r\in [0,1]$, $1-\delta\leq y(r)$, so $\gamma(r)\notin V_{3}$ and $\delta<\frac{7}{10}\leq y(r)-s$. Thus $\gamma(r)\notin V_{1}$. Since $z\notin V_{2}$ and $1-\delta < t$, we have that either $\delta \leq \vert p-f(q)\vert$ or $\delta \leq s$, in both cases it is clear that $\gamma(r)\notin V_{2}$. This completes the proof that for each $r\in [0,1]$, $\gamma(r)\in K$. Therefore $A_{3}\subset K$.

We apply Case 1 to the point $z_{0}=((p,s),(q,1-\delta))$, so there exists a subcontinuum $A_{4}$ of $K$ such that $z_{0},((p,0),(q,\frac{1}{2}))\in A_{4}$. Define $A=A_{3}\cup A_{4}$. Then $A$ has the required properties.  

This finishes the proof of Claim 1.

By the symmetry of the roles of both coordinates in the definition of $V_{0}$, we obtain that the following claim also holds.

{\bf Claim 2.} Let $z=((p,s),(q,t))\in K$ be such that $t\leq s$. Then there
exists a subcontinuum $A$ of $K$ such that $z,((p,\frac{1}{2}),(q,0))\in A$. 

{\bf Claim 3.} Let $z=((p,s),(q,t))\in K$. Then there exists a subcontinuum $B$ of $\sigma(K)$ such that $\sigma(z)\in B$ and $B\cap \sigma(M)\neq \emptyset$. 

The proof for the case $s\leq t$ follows from Claim 1. So  we may suppose that $t\leq s$. By Claim 2, there exists a subcontinuum $A_{4}$ of $K$ such that $z,((p,\frac{1}{2}),(q,0))\in
A_{4}$. Let $A_{5}=\{((p,\frac{1}{2}+r),(q,r)):r\in [0,\frac{1}{2}]\}$. Clearly, $((p,\frac{1}{2}),(q,0)),((p,1),(q,\frac{1}{2}))\in
A_{5}$ and $A_{5}\subset K$. Let $A=A_{4}\cup A_{5}$. Then $A$ is a subcontinuum of $K$, $z\in A$ and $\sigma((p,1),(q,\frac{1}{2}))=\sigma((f(p),0),(q,\frac{1}{2}))\in \sigma(A)\cap\sigma(M)$. Hence $B=\sigma(A)$ satisfies the required properties.

{\bf Claim 4.} Let $z=((p,0),(q,\frac{1}{2}))\in M$. Then there exists a subcontinuum $E$ of $\sigma(K)$ such that $\sigma(z),\sigma((f(p),0),(f(q),\frac{1}{2}))\in E$. 

In order to prove the existence of $E$, let $A=\{((p,r),(q,\frac{1}{2}+r)):r\in [0,\frac{1}{2}]\}\cup\{((p,\frac{1}{2}+r),(f(q),r)):r\in[0,\frac{1}{2}]\}$. Since $\varphi(q,1)=\varphi(f(q),0)$ and $\varphi(p,1)=\varphi(f(p),0)$, we obtain that $E=\sigma(A)$ satisfies the required properties. Hence, Claim 4 is proved. 

Fix a point $(p_{0},q_{0})\in C^{2}$ such that $(p_{0},q_{0})$ has a dense orbit under $f\times f$.
By Claim 4, there exists a sequence $B_{1},B_{2},\ldots$ of subcontinua of $\sigma(K)$ such that for each $n\in \mathbb{N}$, $\sigma((f^{n-1}(p_{0}),0),(f^{n-1}(q_{0}),\frac{1}{2})),\sigma((f^{n}(p_{0}),0),(f^{n}(q_{0}),\frac{1}{2}))\in
B_{n}$. Then the set $B=B_{1}\cup B_{2}\cup\cdots$ is a connected subset of $\sigma(K)$ and $\sigma((p_{0},0),(q_{0},\frac{1}{2}))\in
B$.

Given a point $z=((p,0),(q,\frac{1}{2}))\in M$, since $(p_{0},q_{0})$ has a dense orbit under $f\times f$, we have that $\sigma(z)\in\sigma(\cl_{X^{2}}(\{((f^{n}(p_{0}),0),(f^{n}(q_{0}),\frac{1}{2})):n\in\mathbb{N}\}))\subset \cl_{X^{2}}(B)$. This proves that $B\cup \sigma(M)$ is a connected subset of $\sigma(K)$. Hence, Claim 3 implies that $\sigma(K)$ is connected. This ends the proof of (a).

In order to prove (b), take $\sigma(z)\in \Delta$, where $z=((p,s),(q,t))$. Then $\varphi(p,s)=\varphi(q,t)$. Thus either $(p,s)=(q,t)$ or ($t=1$ and $(p,s)=(f(q),0))$ or ($s=1$ and $(q,t)=(f(p),0)$). In the first case, $z\in V_{1}$, in the second $z\in V_{2}$ and in the third one, $z\in V_{3}$. In any case, $z\notin K$. We have shown that $\Delta\subset X^{2}\setminus \sigma(K)$.

Now we prove that $X^{2}\setminus \sigma(K)\subset U$. Take an element $w=\sigma(z)$, where $z=((p,s),(q,t))\notin K$. Then $z\in V_{1}\cup V_{2}\cup V_{3}$. We consider three cases. 

{\bf Case 1.} $z\in V_{1}$.

In this case, $\vert p-q\vert<\delta$ and $\vert t-s\vert<\delta$, and by the definition of $\delta$, $\sigma(z)\in U$.

{\bf Case 2.} $z\in V_{2}$.

In this case, $\vert p-f(q)\vert<\delta, s<\delta$
and $1-\delta<t$. Then $\rho((p,s),(p,0))<\delta$, $\rho((p,0),(f(q),0))<\delta$ and $\rho((q,1),(q,t))<\delta$. So, $D(\varphi(p,s),\varphi(p,0))<\frac{\varepsilon}{3}$, $D(\varphi(p,0),\varphi(f(q),0))<\frac{\varepsilon}{3}$ and $D(\varphi(q,1),\varphi(q,t))<\frac{\varepsilon}{3}$. Since $\varphi(f(q),0)=\varphi(q,1)$, we have that $D(\varphi(p,s),\varphi(q,t))<\varepsilon$. By the choice of $\varepsilon$, we conclude that $\sigma(z)=(\varphi(p,s),\varphi(q,t))\in U$.

{\bf Case 3.} $z\in V_{3}$.

This case is similar to Case 2.
This completes the proof that $X^{2}\setminus \sigma(K)\subset U$.
Therefore (b) holds.

Finally define $V=X^{2}\setminus \sigma(K)$. Then $V$ is an open subset of $X^{2}$ such that $\Delta\subset V\subset U$, and $X^{2}\setminus V$ is connected. This finishes the proof that $\Delta$ is colocally connected in $X^{2}$.
\end{proof}

\end{document}